\documentclass[10pt]{amsart}

\usepackage{amsmath, amssymb, amsthm}
\usepackage{enumerate}
\usepackage{graphicx}
\usepackage{hyperref}
\usepackage{tikz}

\usetikzlibrary{positioning}
\usetikzlibrary{arrows,automata}

\tikzset{
    position/.style args={#1:#2 from #3}{
        at=(#3.#1), anchor=#1+180, shift=(#1:#2)
    }
}

\newtheorem{thm}{Theorem}
\newtheorem{lemma}[thm]{Lemma}
\newtheorem{prop}[thm]{Proposition}

\theoremstyle{definition}
\newtheorem{defn}{Definition}
\newtheorem{example}[thm]{Example}
\newtheorem*{rem}{Remark}
\newtheorem*{question}{Question}

\theoremstyle{remark}

\begin{document}

\author[M. P. Cohen]{Michael P. Cohen}
\email[1]{mcohen@carleton.edu}

\author[T. Johnson]{Todd Johnson}
\email[2]{johnsont2@carleton.edu}

\author[A. Kral]{Adam Kral}
\email[3]{krala@carleton.edu}

\author[A. Li]{Aaron Li}
\email[4]{lia2@carleton.edu}

\author[J. Soll]{Justin Soll}
\email[5]{sollj2@carleton.edu}

\address{Department of Mathematics and Statistics,
Carleton College,
One North College Street,
Northfield, MN 55057}



\title{A Kuratowski closure-complement variant whose solution is independent of ZF}

\begin{abstract}  We pose the following new variant of the Kuratowski closure-complement problem:  How many distinct sets may be obtained by starting with a set $A$ of a Polish space $X$, and applying only closure, complementation, and the $d$ operator, as often as desired, in any order?  The set operator $d$ was studied by Kuratowski in his foundational text \textit{Topology: Volume I}; it assigns to $A$ the collection $dA$ of all points of second category for $A$.  We show that in ZFC set theory, the answer to this variant problem is $22$.  In a distinct system equiconsistent with ZFC, namely ZF+DC+PB, the answer is only $18$.
\end{abstract}

\maketitle

\section*{Introduction}

Kuratowski's \textit{closure-complement theorem}, a result of his 1922 thesis, states that at most $14$ distinct sets are obtainable by applying the operations of closure and complementation to any particular initial set $A$ in any topological space $X$.  The algebraic result underlying this theorem is that the monoid generated by the set operators $k$ (closure) and $c$ (complement) has cardinality $\leq 14$.  This surprising and amusing result has inspired a substantial literature of generalizations and variants; see for example \cite{Gaida_Eremenko_1974a}, \cite{Gardner_Jackson_2008a}, \cite{Sherman_2010a}, \cite{Banakh_2018a}, \cite{CCGS_2020a} or visit Bowron's website \textit{Kuratowski's Closure-Complement Cornucopia} \cite{Bowron_2012a} for a comprehensive list of relevant literature.

The purpose of this note is to give an example of a natural variant of the Kuratowski closure-complement problem, whose solution turns out to be independent of ZF set theory.  To pose the problem, we let $X$ be a topological space.  Given a subset $A\subseteq X$, we say that a point $p\in X$ is a \textit{point of second category} for $A$ if whenever $U\subseteq X$ is an open neighborhood of $p$, we have $U\cap A$ nonmeager in $X$.  Then, we define

\begin{center} $dA=\{p\in X: p$ is a point of second category for $A\}$.
\end{center}
\vspace{.3cm}

The operator $d$ was apparently first defined by Kuratowski himself in his foundational text \textit{Topology Vol. 1} (\cite{Kuratowski_1966a}, first edition 1933) and is associated with the application of Baire category methods in general topology.  The operator has handy applications in important descriptive set theoretic results, especially as it appears in Pettis's lemma which states that if $A,B\subseteq X$ have the \textit{Baire property} (Definition \ref{defn_BP}), then $AB\supseteq id(A)id(B)$ (where $i$ denotes the topological interior operator), and thus $i(AB)$ is nonempty \cite{Pettis_1950a}.  This lemma implies that every Borel-measurable homomorphism between Polish (i.e., separable and completely metrizable) topological groups is automatically continuous (see \cite{Rosendal_2009a} for an admirable survey of this and many related results).  We ask the following.

\begin{question}  How many distinct sets may be obtained by starting with a subset $A$ of a Polish space $X$, and applying only the operators $k$, $c$, and $d$, as often as desired, in any order?  Equivalently, what is the maximal cardinality of the monoid of set operators generated by $k$, $c$, and $d$?
\end{question}

For the remainder of this paper, $X$ denotes an arbitrary Polish (separable completely metrizable) space, $k$ and $i$ the closure and interior operators on $X$ respectively, and $c$ the complementation operator.  We recall the DeMorgan's law for interiors and closures which states that $kc=ci$, or equivalently that $ic=ck$.  We let $\mathcal{KD}$ denote the monoid of set operators on $X$ generated by $k$, $c$, and $d$.  We first answer the question in the traditional domain, where we assume the usual ZF axioms plus the Axiom of Choice (AC).

\begin{thm}[ZFC] \label{thm1}  The cardinality of $\mathcal{KD}$ is $\leq 22$.  Moreover, if $X=\mathbb{R}$ with the usual topology, then there exists a set $A\subseteq\mathbb{R}$ for which $\#\{oA:o\in\mathcal{KD}\}=22$.
\end{thm}

On the other hand, weak forms of AC are not sufficient to obtain the solution above.  We denote by DC the Axiom of Dependent Choice, which is equivalent over ZF to the Baire Category Theorem.  We denote by PB the axiom that ``every subset of every Polish space has the Baire property,'' and we recall the definition of the Baire property below.

\begin{defn} \label{defn_BP}  A set $A\subseteq X$ has the \textit{Baire property} if there exists an open set $U\subseteq X$ for which the symmetric difference $A\Delta U=(A-U)\cup(U-A)$ is a meager set.
\end{defn}

In the seminal paper \cite{Solovay_1970a}, Solovay showed that if ZF is consistent with the existence of an inaccessible cardinal, then ZF+DC+PB is consistent.  In \cite{Shelah_1984a}, Shelah improved this result to show that ZFC and ZF+DC+PB are equiconsistent axiom systems.  Our second theorem below shows that the solution to this natural extension of the Kuratowski problem differs in this alternative axiom system, and thus the solution is independent of ZF.

\begin{thm}[ZF+DC+PB] \label{thm2}  The cardinality of $\mathcal{KD}$ is $\leq 18$.  Moreover, if $X=\mathbb{R}$ with the usual topology, then there exists a set $A\subseteq\mathbb{R}$ for which $\#\{oA:o\in\mathcal{KD}\}=18$.
\end{thm}

\section*{Preliminaries and Sets with the Baire Property}

First we establish the basic properties of the operator $d$, most of which are observed without proof in \cite{Kuratowski_1966a} 4.IV.

\begin{lemma}[ZF+DC] \label{lemma_main}  Let $X$ be a Polish space, and $A,B\subseteq X$.
\begin{enumerate}[(a)]
		\item $A\subseteq B$ implies $dA\subseteq dB$.
		\item $dA$ is closed and therefore $dA\subseteq kA$.
		\item $A$ open implies $dA=kA$.
		\item $d(A\cup B)=dA\cup dB$.
		\item $A-dA$ is meager.
		\item $A$ is meager in $X$ if and only if $dA=\emptyset$.
		\item $ddA=dA$.
		\item $dkA=kikA$.
		\item $kidA=dA$.
\end{enumerate}
\end{lemma}

\begin{proof} \textit{(a)}  Immediate from the definition of $d$.

\textit{(b)}  Suppose $p\in X$ is a limit point of $dA$.  Given an arbitrary open neighborhood $U$ of $p$, it means there is $x\in dA$ with $x\in U$.  Since $x$ is a point of second category for $A$, $U\cap A$ is nonmeager in $X$, whence $p\in dA$.

\textit{(c)}  We have $dA\subseteq kA$ by (c).  Conversely if $p\in kA$, then each open neighborhood $U$ of $p$ will satisfy $U\cap A\neq\emptyset$.  Assuming $A$ is open, then $U\cap A$ is a nonempty open set and hence nonmeager by the Baire Category Theorem (equivalent to DC).  Thus $p\in dA$.

\textit{(d)}  By (a), we have $dA\cup dB\subseteq d(A\cup B)$.  Conversely suppose $p\notin dA\cup dB$.  Then there are open neighborhoods $U,V$ of $P$ so that $U\cap A$ is meager and $V\cap B$ is meager.  But then $U\cap V$ is an open neighborhood of $p$ whose intersection with $A\cup B$ is meager, because $(U\cap V)\cap(A\cup B)\subseteq (U\cap A)\cup (V\cap B)$, where the latter is a union of two meager sets and hence meager.

\textit{(e)}  Since $X$ is Polish, we may find a countable family of open sets $\{B_i:i\in\mathbb{N}\}$ which comprise a basis for the topology of $X$.  Form the subcollection $\mathcal{C}=\{B_i:B_i\cap A$ is meager$\}$.  For each $x\in A-dA$, since $x$ is not a point of second category for $A$, we may find a neighborhood $B_i$ for which $x\in B_i$ and $B_{i}\cap A$ is meager, so $B_i\in\mathcal{C}$.  This shows $A-dA\subseteq\bigcup_{B_i\in\mathcal{C}}B_i\cap A$, so $A-dA$ is a subset of a countable union of meager sets and hence meager.

\textit{(f)}  If $A$ is meager, then $dA=\emptyset$ immediately from the definition of $d$.  Conversely, if $dA=\emptyset$, then $A=A-dA$ and $A$ is meager by (e).

\textit{(g)}  By (b), $ddA\subseteq kdA=dA$.  By (a), (d), (e), and (f), $dA\subseteq d((A-dA)\cup dA)=d(A-dA)\cup ddA=ddA$.

\textit{(h)}  Note that $kA-ikA$ is a closed nowhere dense set and hence meager.  So by (c), (d) and (f), we have $dkA=d((kA-ikA)\cup ikA)=d(kA-ikA)\cup dikA=\emptyset\cup kikA=kikA.$

\textit{(i)}  By (b), (g), and (h), $kidA=kikdA=dkdA=ddA=dA$.
\end{proof}

\begin{lemma}[ZF+DC] \label{lemma_BP}  Let $X$ be a Polish space, and $A\subseteq X$.  Then the following are equivalent.
\begin{enumerate}[(a)]
		\item $A$ has the Baire property.
		\item $dA-A$ is meager.
		\item $idcA=cdA$.
		\item $idA=cdcA$.
		\item $dA=cidcA$.
		\item $dcA=kcdA$.
\end{enumerate}
\end{lemma}

\begin{proof} \textit{(a $\Rightarrow$ b)}  Assume $A$ has the Baire property, and find $U\subseteq X$ open so that $M=A\Delta U$ is a meager set.  We have $A=M\Delta U$, and therefore $dA=d((M-U)\cup(U-M))\cup\emptyset=d(M-U)\cup d(U-M)\cup d(M\cap U)=\emptyset\cup d((U-M)\cup (M\cap U))=dU$ by Lemma \ref{lemma_main} (d) and (f).  So $dA=dU$, and by applying the same argument, we obtain $dcA=dcU$, because $M=(cA)\Delta(cU)$.

Now we have 

\begin{align*}
dA-A &= (dA-A-dcA)\cup((dA-A)\cap dcA)\\
&\subseteq (cA-dcA)\cup(dA\cap dcA)\\
&= (cA-dcA)\cup(dU\cap dcU).
\end{align*}
\vspace{.3cm}

The set $cA-dcA$ is meager by Lemma \ref{lemma_main} (e), and the set $dU\cap dcU\subseteq kU\cap kcU=kU-U$ is nowhere dense.  So $dA-A$ is a subset of the union of two meager sets, hence meager.

\textit{(b $\Rightarrow$ a)}  Assume $dA-A$ is meager.  Set $U=idA$, so $U$ is an open subset of $X$.  We have $U-A$ meager because $U-A\subseteq dA-A$.  Also, $A-U=(A-dA)\cup(dA-U)$, where $A-dA$ is meager by Lemma \ref{lemma_main} (e), and $dA-U=dA-idA$ is closed nowhere dense, so $A-U$ is meager.  Therefore $A\Delta U$ is meager, and we conclude $A$ has the Baire property.

\textit{(b $\Rightarrow$ c)}  Assume $dA-A$ is meager.  Applying Lemma \ref{lemma_main} (b), (c), (d), (e), (f), and (i), as well as the DeMorgan's law for interior/closure, we have

\begin{align*}
idcA &= id((cA-dA)\cup(cA\cap dA))\\
&= i[d(cA-dA)\cup d(cA\cap dA)]\\
&= i[d(cdA-A)\cup\emptyset\cup d(dA-A)]\\
&= i[d(cdA-A)\cup d(A-dA)\cup\emptyset]\\
&= id((cdA-A)\cup (cdA\cap A))\\
&= idcdA\\
&= ikcdA\\
&= ckidA\\
&= cdA.
\end{align*}
\vspace{.2cm}

\textit{(c $\Rightarrow$ d)}  Assume $idcA=cdA$.  Note that $idcA\cup idA=idX=X$ by Lemma \ref{lemma_main} (d), so $idA\supseteq cidcA=kcdcA\supseteq cdcA$.  Conversely, $cdA\cup cdcA=idcA\cup cdcA=i(dcA\cup cdcA)=iX=X$, so $cdcA\supseteq ccdA=dA\supseteq idA$.  So $idA=cdcA$.

\textit{(d $\Rightarrow$ e)}  Assume $idA=cdcA$.  Taking the closure of both sides, we have $dA=kidA=kcdcA=cidcA$.

\textit{(e $\Rightarrow$ f)}  Assume $dA=cidcA$.  Since $dA\cup dcA=d(A\cup cA)=X$ by Lemma \ref{lemma_main} (d), we have $dcA\supseteq cdA$, and since $dcA$ is closed (Lemma \ref{lemma_main} (b)) we also have $dcA\supseteq kcdA$.  Conversely, we have $cidcA\cup kcdA=dA\cup kcdA\supseteq dA\cup cdA=X$, so $kcdA\supseteq ccidcA=idcA$.  Since $kcdA$ is closed, we also have $kcdA\supseteq kidcA=dcA$ by Lemma \ref{lemma_main} (i).  So $dcA=kcdA$.

\textit{(f $\Rightarrow$ b)}  Assume $dcA=kcdA$.  To show $dA-A$ is meager, by Lemma \ref{lemma_main} (f) it suffices to show that $d(dA-A)=\emptyset$.  We have $d(dA-A)=d(dA\cap cA)\subseteq (ddA\cap dcA)$ by Lemma \ref{lemma_main} (a).  Therefore by assumption, $d(dA-A)\subseteq dA\cap kcdA$, so $id(dA-A)\subseteq idA\cap ikcdA=idA\cap ckidA=idA\cap cdA\subseteq dA\cap cdA=\emptyset$.  Therefore $d(dA-A)=kid(dA-A)=k\emptyset=\emptyset$.
\end{proof}

\begin{proof}[Proof of Theorem \ref{thm2}]  We work in ZF+DC+PB.  We denote by $\mathcal{K}_0$ the monoid of set operators generated by $k$ and $i$, and by $\mathcal{KD}_0$ the monoid of set operators generated by $k$, $i$, and $d$.  Since $i=ckc$, $\mathcal{K}_0$ and $\mathcal{KD}_0$ are submonoids of $\mathcal{KD}$.  The proof of Kuratowski's closure-complement theorem relies on observing that $kiki=ki$ and $ikik=ik$, and therefore

\begin{center} $\mathcal{K}_0=\{e, i, k, ki, ik, iki, kik\}$,
\end{center}
\vspace{.2cm}

\noindent where $e$ denotes the identity operator.  $\mathcal{K}_0$ is often called the monoid of \textit{even operators} in the closure-complement problem.  Using the tools of Lemma \ref{lemma_main}, we compute the (no more than seven) members of $\mathcal{K}_0d$:  $d$, $id$, $kd=d$, $kid=d$, $ikd=id$, $ikid=id$, and $kikd=d$.  So $\mathcal{KD}_0\supseteq \mathcal{K}_0d=\{d,id\}$.  In fact, we have the following set equality:

\begin{center} $\mathcal{KD}_0=\{e, i, k, ki, ik, iki, kik, d, id\}$
\end{center}
\vspace{.2cm}

\noindent which is easily verified by using Lemma \ref{lemma_main} to check that $i\mathcal{KD}_0\subseteq\mathcal{KD}_0$, $k\mathcal{KD}_0\subseteq\mathcal{KD}_0$, and $d\mathcal{KD}_0\subseteq\mathcal{KD}_0$.  So there are nine even operators in $\mathcal{KD}_0$.  Applying $c$ to the left (or right) of $\mathcal{KD}_0$ yields nine more operators, the odd operators, as depicted in Figure \ref{fig1}.

\begin{figure}[ht]
\begin{center} \begin{tabular}{|c|c|}
\hline
Even Operators & Odd Operators\\
\hline
$e$ & $c$\\
$i$ & $ci=kc$\\
$k$ & $ck=ki$\\
$ki$ & $cki=ikc$\\
$ik$ & $cik=kic$\\
$iki$ & $ciki=kikc$\\
$kik$ & $ckik=kikc$\\
$d$ & $cd=idc$\\
$id$ & $cid=kcd=dc$\\
\hline
\end{tabular}
\end{center}
\caption{Operators in $\mathcal{KD}=\mathcal{KD}_0\cup c\mathcal{KD}_0$.}
\label{fig1}
\end{figure}
\vspace{.2cm}

The equalities in the last two entries in the table of Figure \ref{fig1} hold because every set in $X$ has the Baire property, allowing us to apply Lemma \ref{lemma_BP} universally.  Noting that $c\mathcal{KD}_0=\mathcal{KD}_0c$ consists of the nine odd operators in the right column, we claim that $\mathcal{KD}=\mathcal{KD}_0\cup c\mathcal{KD}_0$, which proves that $\#\mathcal{KD}\leq 18$, as claimed.  To see this, simply check the following set equalities which show that $\mathcal{KD}_0\cup c\mathcal{KD}_0$ is closed under multiplication from the left by the generating operators $k$, $c$, and $d$:

\begin{center}
$k(\mathcal{KD}_0\cup c\mathcal{KD}_0) = k\mathcal{KD}_0\cup kc\mathcal{KD}_0=\mathcal{KD}_0\cup ci\mathcal{KD}_0=\mathcal{KD}_0\cup c\mathcal{KD}_0$;\\
$c(\mathcal{KD}_0\cup c\mathcal{KD}_0) = c\mathcal{KD}_0\cup cc\mathcal{KD}_0=\mathcal{KD}_0\cup c\mathcal{KD}_0$;\\
$d(\mathcal{KD}_0\cup c\mathcal{KD}_0) = d\mathcal{KD}_0\cup dc\mathcal{KD}_0=\mathcal{KD}_0\cup cid\mathcal{KD}_0=\mathcal{KD}_0\cup c\mathcal{KD}_0$.
\end{center}
\vspace{.3cm}

To prove the second statement of the theorem, we can take for example

\begin{center} $A=(1,2)\cup(2,3)\cup\{4\}\cup[(5,6)\cap\mathbb{Q}]\cup[(6,7)\cap(\mathbb{R}-\mathbb{Q})]$
\end{center}
\vspace{.3cm}

\noindent and see that application of the $9$ even operators in $\mathcal{KD}_0$ yields $9$ distinct sets as depicted in Figure \ref{fig4}.  Taking complements, we get $18$ distinct sets from the $18$ distinct operators in $\mathcal{KD}$.

\begin{figure}[ht]
\begin{center} \begin{tabular}{|c|c|}
\hline
$eA$ & $(1,2)\cup(2,3)\cup\{4\}\cup[(5,6)\cap\mathbb{Q}]\cup[(6,7)\cap(\mathbb{R}-\mathbb{Q})]$\\
\hline
$iA$ & $(1,2)\cup(2,3)$\\
\hline
$kA$ & $[1,3]\cup\{4\}\cup[5,7]$\\
\hline
$kiA$ & $[1,3]$\\
\hline
$ikA$ & $(1,3)\cup(5,7)$\\
\hline
$ikiA$ & $(1,3)$\\
\hline
$kikA$ & $[1,3]\cup[5,7]$\\
\hline
$dA$ & $[1,3]\cup[6,7]$\\
\hline
$idA$ & $(1,3)\cup(6,7)$\\
\hline
\end{tabular}
\end{center}
\caption{Even operators applied to $A$.}
\label{fig4}
\end{figure}
\vspace{.2cm}
\end{proof}

\section*{Vitali Sets and Distinguishing Words Under AC}

Recall that a \textit{Vitali set} is a subset $V\subseteq\mathbb{R}$ consisting of exactly one representative from each coset in the quotient group $\mathbb{R}/\mathbb{Q}$.  Vitali sets can be constructed by invoking the Axiom of Choice, and do not have the Baire property.

\begin{prop}[ZFC] \label{prop1}  Let $W_0\subseteq\mathbb{R}$ be open.  Then there exists a Vitali set $V$ such that $dV=kW_0$.
\end{prop}

\begin{proof}  Let $\alpha$ be an arbitrary irrational real number, and let $H=\langle\mathbb{Q},\alpha\rangle$ be the additive subgroup of $\mathbb{R}$ generated by $\mathbb{Q}$ and $\alpha$, so $H=\{q+n\alpha:q\in\mathbb{Q},n\in\mathbb{Z}\}$.  Let $V^1$ be a set consisting of exactly one representative from each coset of $\mathbb{R}/H$.  For each $n\in\mathbb{Z}$, we define the sets

\begin{center} $P_n=\{v+n\alpha+\mathbb{Q}:v\in V^1\}\subseteq\mathbb{R}/\mathbb{Q}$
\end{center}
\vspace{.2cm}

\noindent and

\begin{center} $R_n=\bigcup P_n\subseteq\mathbb{R}$.
\end{center}
\vspace{.2cm}

We first claim that $\bigcup_{n\in\mathbb{Z}}P_n=\mathbb{R}/\mathbb{Q}$.  The left-to-right inclusion is by definition.  For the right-to-left inclusion, consider an arbitrary coset $x+\mathbb{Q}$ in $\mathbb{R}/\mathbb{Q}$.  There exists a unique element $v\in V^1$ for which $x+H=v+H$, i.e. $x-v\in H$.  Therefore we may write $x-v=q+n\alpha$ for some $q\in\mathbb{Q}$ and some $n\in\mathbb{Z}$.  But then $x-q=v+n\alpha$, whence $x+\mathbb{Q}=v+n\alpha+\mathbb{Q}\in P_n$.  So $\mathbb{R}/\mathbb{Q}\subseteq\bigcup_{n\in\mathbb{Z}}P_n$ as claimed.

Moreover, the sets $P_n$ are pairwise disjoint.  For if $P_n\cap P_m\neq\emptyset$, it means that there are $v,w\in V^1$ for which $v+n\alpha+\mathbb{Q}=w+m\alpha+\mathbb{Q}$, i.e. $v+n\alpha=w+m\alpha+q$ for some $q\in\mathbb{Q}$.  But then $v-w=(m-n)\alpha+q\in H$, so $v=w$ by construction of $V^1$.  In turn, this implies $(m-n)\alpha=-q\in\mathbb{Q}$, and hence $n=m$ since $\alpha$ is irrational.

The preceding two paragraphs imply that the family $\{P_n:n\in\mathbb{Z}\}$ forms a partition of $\mathbb{R}/\mathbb{Q}$.  Consequently, the sets $R_n=n\alpha+R_0$ comprise a partition of $\mathbb{R}$, and we conclude that each set $R_n$ is nonmeager in $\mathbb{R}$.

Next, let $\{B_n:n\in\mathbb{Z}\}$ be a countable basis of open sets for the topology on $W_0$.  For each $n\in\mathbb{Z}$, each coset $v+n\alpha+\mathbb{Q}$ in $P_n$ is dense in $\mathbb{R}$, and hence meets $B_n$.  So let $V_n$ be a set consisting of exactly one element chosen from each intersection $(v+n\alpha+\mathbb{Q})\cap B_n$ ($v\in V^1$).  Then $V=\bigcup_{n\in\mathbb{Z}}V_n$ consists of exactly one representative from each distinct coset in $\bigcup_{n\in\mathbb{Z}}P_n=\mathbb{R}/\mathbb{Q}$, so $V$ is a Vitali set.

If $x\in kW_0$, and $U\subseteq\mathbb{R}$ is an arbitrary open neighborhood of $x$, then there exists $n\in\mathbb{Z}$ with $B_n\subseteq U$.  But then $V_n\subseteq V\cap B_n\subseteq V\cap U$, and $R_n=\bigcup P_n=\bigcup_{z\in V_n}\bigcup_{q\in\mathbb{Q}}(z+q)=\bigcup_{q\in\mathbb{Q}}(V_n+q)$.  Since $R_n$ is nonmeager, it follows that $V_n$ is nonmeager and hence $V\cap U$ is nonmeager.  So $x\in dV$ and we have shown $kW_0\subseteq dV$.  Conversely, $V\subseteq W_0$ so $dV\subseteq kV\subseteq kW_0$, and the proposition is proven.
\end{proof}

\begin{prop}[ZFC] \label{prop2}  Let $W_0\subseteq W_1\subseteq\mathbb{R}$ such that $W_0$ and $W_1$ are both open.  Then there exists a Vitali set $V$ such that $dV=kW_0$ and $kV=kW_1$.
\end{prop}

\begin{proof}  Using Proposition \ref{prop1}, start with a Vitali set $V_0$ satisfying $dV_0=kW_0$.  Let $\{v_n:n\in \mathbb{N}\}$ be an arbitrary sequence of distinct elements in $V_0$, and let $\{C_n:n\in\mathbb{N}\}$ be a countable basis of open sets for the topology on $W_1$.  For each $n$, let $w_n\in (v_n+\mathbb{Q})\cap C_n$.  Define $V=(V_0-\{v_n:n\in \mathbb{N}\})\cup\{w_n:n\in\mathbb{N}\}$, so $V$ is a Vitali set.

Since $V\Delta V_0$ is countable, hence meager, we have $dV=dV_0=kW_0$.  We also have $kV\subseteq kW_1$ since $V\subseteq W_1$, and $kW_1\subseteq kV$ since $V$ is dense in $W_1$ by construction.
\end{proof}

\begin{proof}[Proof of Theorem \ref{thm1}]  Working in ZFC, Lemma \ref{lemma_main} still holds, so $\mathcal{KD}$ consists of at least the $18$ operators in $\mathcal{KD}_0\cup c\mathcal{KD}_0$ (see Figure \ref{fig1}).  However, the identities in Lemma \ref{lemma_BP} do not apply to every subset of $X$, and thus in general we do not have $idc=cd$, $id=cdc$, $d=cidc$, or $dc=kcd$.  To see that these equalities fail, we apply Proposition \ref{prop2} and construct $V$ a Vitali set satisfying $dV=[8,9]$ and $kV=[8,10]$. 

The complement $cV$ has the following property: for every open set $U$ in $\mathbb{R}$, the intersection $U\cap cV$ contains a representative from each coset of $\mathbb{Q}$ (in fact infinitely many representatives).  Thus $\mathbb{R}\subseteq \bigcup_{q\in\mathbb{Q}}q+(U\cap cV)$, so $\mathbb{R}$ is covered by countably many translates of $U\cap cV$.  This implies $U\cap cV$ is nonmeager.

The preceding paragraph implies $dcV=kcV=\mathbb{R}$.  So $V$ distinguishes additional operators in the monoid $\mathcal{KD}$, as depicted in the table below:

\begin{figure}[ht]
\begin{center}
\begin{tabular}{|rl||rl|}
\hline
$idcV=$ & $\mathbb{R}$ & $cdV=$ & $\mathbb{R}-[8,10]$\\
\hline
$idV=$ & $(8,9)$ & $cdcV=$ & $\emptyset$\\
\hline
$dV=$ & $[8,9]$ & $cidcV=$ & $\emptyset$\\
\hline
$dcV=$ & $\mathbb{R}$ & $kcdV=$ & $\mathbb{R}-(8,10)$\\
\hline
\end{tabular}
\end{center}
\label{fig2}
\end{figure}
\vspace{.2cm}

Thus we claim that in ZFC, we have

\begin{center} $\mathcal{KD}=\{e, i, k, ki, ik, iki, kik, d, id, c, ci, ck, cki, cik, ciki, ckik, cd, cid, dc, idc, cdc, cidc\}$.
\end{center}
\vspace{.2cm}

To verify the claim, one must check that $\mathcal{KD}$ is invariant under both left and right multiplication by $k$, $c$, and $d$, and we leave the task to the reader using Lemma \ref{lemma_main}.  So $\#\mathcal{KD}\leq 22$.

For an example of an explicit initial set in $\mathbb{R}$ which distinguishes all $22$ operators, we give

\begin{center} $A=(1,2)\cup(2,3)\cup\{4\}\cup[(5,6)\cap\mathbb{Q}]\cup[(6,7)\cap(\mathbb{R}-\mathbb{Q})]\cup V$, 
\end{center}
\vspace{.2cm}

\noindent where $V$ is as in the first paragraph of the proof.
\end{proof}

\begin{rem}  Examining the $22$ operators in $\mathcal{KD}$ in Theorem \ref{thm1}, we may regard them as either \textit{even} or \textit{odd} depending on the number of instances of the $c$ operator in the reduced word.  So we find $11$ even operators and $11$ odd operators.  However, in this case the monoid $\mathcal{KD}_0$ generated by $k$, $i$, and $d$ does not yield all even operators, nor does either of the sets $c\mathcal{KD}_0$ or $\mathcal{KD}_0$ consist of all odd operators.  (Contrast with the situation in the original closure-complement problem, and in Theorem \ref{thm2}.)
\end{rem}

\section*{Partial Orderings and Other Addenda}

The monoid $\mathcal{KD}$ admits a natural partial ordering defined by the rule $o_1\leq o_2$ if and only if $o_1A\subseteq o_2A$ for every set $A\subseteq X$.  The partial ordering on $\mathcal{K}_0$ (the monoid generated by $k$ and $i$) has been diagrammed by various authors; see for instance \cite{Gardner_Jackson_2008a}.  In general if $o_1\leq o_2$ then $io_1\leq io_2$, $ko_1\leq ko_2$, and $co_1\geq co_2$.  We observe also the following proposition.

\begin{prop}[ZF+DC]  The following relations hold among even operators in $\mathcal{KD}$.
\begin{enumerate}[(a)]
		\item $d\leq kik$;
		\item $iki\leq cdc$;
		\item $cdc\leq id$;
		\item $cdc\leq cidc$; 
		\item $ki\leq cidc$; and
		\item $cidc\leq d$.
\end{enumerate}
\end{prop}

\begin{proof} \textit{(a)}  Since $d\leq k$, we have $d=kid\leq kik$.

\textit{(b)} By (a) we have $dc\leq kikc$, and hence $cdc\geq ckikc=iki$.

\textit{(c)}  Let $A\subseteq X$ be arbitrary.  If $p\in cdcA$, then $p$ has an open neighborhood $U$ for which $U\cap cA$ is meager.  Given arbitrary $x\in U$ and an arbitrary open neighborhood $V$ of $x$, we can observe that $(U\cap V)\cap cA\subseteq U\cap cA$ is meager, and hence $(U\cap V)\cap A$ is nonmeager, because $U\cap V$ is nonmeager (being an open set).  So $V\cap A$ is nonmeager, which implies $x\in dA$.  Therefore $U\subseteq dA$ which implies $p\in idA$ and $cdcA\subseteq idA$.

\textit{(d)}  Since $id\leq d$ we have $idc\leq dc$ and therefore $cdc\leq cidc$.

\textit{(e)}  Apply $k$ to the left side of the inequality in (b).

\textit{(f)}  Apply $k$ to the left side of the inequality in (c).
\end{proof}

Combining the preceding inequalities with the known ordering on $\mathcal{K}_0$, we obtain the partial ordering on the even operators of $\mathcal{KD}$ presented in Figure \ref{fig3}.  For each pair of even operators $o_1,o_2\in\mathcal{KD}$ not connected by an arrow in the diagram, the reader may verify that $o_1A\not\subseteq o_2A$ where $A$ is one of the sets given in the proofs of Theorems \ref{thm1} and \ref{thm2}.  Thus the diagram is complete.

\begin{figure}[ht]
\begin{minipage}{0.48\textwidth}
\begin{center}
\begin{tikzpicture}[->,>=stealth',shorten >=1pt,auto,node distance=2.5cm,
        scale = .75,transform shape, state without output/.append style={draw=none}]

\node[state] (i) [] {$i$};
\node[state] (iki) [position=60:{1.4cm} from i] {$iki$};
\node[state] (cdc) [position=0:{0.8cm} from iki] {$cdc$};
\node[state] (id) [position=0:{0.8cm} from cdc] {$id$};
\node[state] (ik) [position=0:{0.8cm} from id] {$ik$};
\node[state] (ki) [position=-45:{0.8cm} from iki] {$ki$};
\node[state] (d) [position=-45:{0.8cm} from id] {$d$};
\node[state] (cidc) [position=-45:{0.8cm} from cdc] {$cidc$};
\node[state] (kik) [position=0:{0.8cm} from d] {$kik$};
\node[state] (k) [position=-70:{0.8cm} from kik] {$k$};
\node[state] (e) [position=-90:{0.7cm} from ki] {$e$};

\path
(i) edge node {} (iki)
(iki) edge node {} (cdc)
(i) edge node {} (e)
(cdc) edge node {} (id)
(cdc) edge node {} (cidc)
(iki) edge node {} (ki)
(e) edge node {} (k)
(ki) edge node {} (cidc)
(ik) edge node {} (kik)
(id) edge node {} (ik)
(id) edge node {} (d)
(d) edge node {} (kik)
(cidc) edge node {} (d)
(kik) edge node {} (k);

\end{tikzpicture}
\end{center}
\end{minipage} \begin{minipage}{0.48\textwidth}
\begin{center}
\begin{tikzpicture}[->,>=stealth',shorten >=1pt,auto,node distance=2.5cm,
        scale = .75,transform shape, state without output/.append style={draw=none}]

\node[state] (i) [] {$i$};
\node[state] (iki) [position=60:{1.4cm} from i] {$iki$};
\node[state] (id) [position=0:{0.8cm} from iki] {$id$};
\node[state] (ik) [position=0:{0.8cm} from id] {$ik$};
\node[state] (ki) [position=-45:{0.8cm} from iki] {$ki$};
\node[state] (d) [position=-45:{0.8cm} from id] {$d$};
\node[state] (kik) [position=0:{0.8cm} from d] {$kik$};
\node[state] (k) [position=-70:{0.8cm} from kik] {$k$};
\node[state] (e) [position=-90:{0.7cm} from ki] {$e$};

\path

(i) edge node {} (iki)
(iki) edge node {} (cdc)
(i) edge node {} (e)
(iki) edge node {} (ki)
(e) edge node {} (k)
(ik) edge node {} (kik)
(id) edge node {} (ik)
(id) edge node {} (d)
(d) edge node {} (kik)
(id) edge node {} (d)
(ki) edge node {} (d)
(kik) edge node {} (k);

\end{tikzpicture}
\end{center}
\end{minipage}
\caption{\textit{Left:} the partial ordering on the $11$ even operators of $\mathcal{KD}$ in ZFC.  \textit{Right:} the partial ordering on the $9$ even operators of $\mathcal{KD}$ in ZF+DC+PB.}
\label{fig3}
\end{figure}
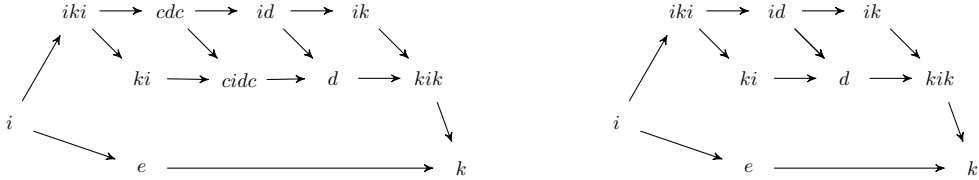

\begin{example}[Another ZF-Independent Problem]  We also consider the problem of the cardinality of the monoid $\mathcal{KFD}=\langle k, c, f, d\rangle$ generated by $k$, $c$, $d$, and the topological frontier operator $f$ defined by $fA=kA\cap kcA$ for all sets $A$.  As one would expect, the cardinality of this monoid also depends on axiomatic assumptions.

The submonoid generated by $k$, $c$, and $f$ has size $\leq 34$, as shown by Gaida and Eremenko in \cite{Gaida_Eremenko_1974a}.  This can be computed using the following identities:   $fff=ff$, $fc=kf=f$, $ffk=fk$, $ifk=0$, $fkik=fki$, and $fiki=fik$, where $0$ denotes the ``empty set operator'' defined by $0A=\emptyset$ for every $A$.  To this list we add the following three identities whose proofs we leave to the reader: $df=kif$, $fid=fd$, $dfk=0$.  We are also interested in cardinality of the submonoid generated by $k, i, f, d$.  We compute the following presentations and cardinalities.

\begin{center}
\begin{tabular}{|c|c|c|p{6cm}|}
\hline
Axiom System & Generators & Cardinality & List of Elements\\
\hline
ZF+DC & $\langle k, i, f, d\rangle$ & $20$ & $\{e, k, i, d, f, ik, fk, ki, fi, fd, id, if, ff,$\newline $kif, kik, fik, 0=ifk, iki, fki, fif\}$\\
\hline
ZF+DC+PB & $\langle k, c, f, d\rangle$ & $40$ & $\{$above$\}\cup\{c, ck, ci, cd, cf, cik, cfk, cki,$\newline $cfi, cfd, cid, cif, cff, ckif, ckik, cfik,$\newline $1=c0, ciki, cfki, cfif\}$\\
\hline
ZFC & $\langle k, c, f, d\rangle$ & $46$ & $\{$above$\}\cup\{dc, idc, cdc, cidc, fdc, cfdc\}$\\
\hline
\end{tabular}
\end{center}
\vspace{.2cm}

The initial set $A$ given in the proof of Theorem \ref{thm1} is sufficient to distinguish the $46$ operators in $\mathcal{KFD}$.
\end{example}

\begin{rem}[Suggestions for Further Projects]  Several variant problems remain open to be solved by an interested party.  For example, how many distinct sets are obtainable using $k$, $i$, $d$, together with one or both of $\cap$ and $\cup$?  (See \cite{Gardner_Jackson_2008a} Section 4 for more information.)

Also, it was shown by Kuratowski that it is possible to obtain infinitely many sets using $k$, $c$, and either $\cap$ or $\cup$.  We believe replacing $k$ with $d$ should yield a finite answer and it may be interesting to compute.

More broadly, the operator $d$ is an example of a \textit{local function} associated to a $\sigma$-ideal $\mathcal{I}$ on a topological space $X$.  A general local function $\ell$ associated to $\mathcal{I}$ assigns to a set $A\subseteq X$ the set $\ell A$ consisting of all points $p\in X$ for which every open neighborhood $U$ of $p$ satisfies $U\cap A\notin\mathcal{I}$.  For $d$, the $\sigma$-ideal in question is the family of meager subsets of $X$.  It may be interesting to study variants of the Kuratowski problem using local functions associated to other $\sigma$-ideals.

Moreover, given a local function $\ell$, the operator $k_\ell$ defined by $k_\ell A=A\cup \ell A$ is an example of a Kuratowski closure operator, which generates a topology finer than the original.  In fact the new topology and the old are \textit{saturated} in the sense that every open set in either has nonempty interior in the other.  There exists some literature on variants of the Kuratowski problem in spaces equipped with multiple topologies (i.e. \textit{polytopological spaces}), including the special case of saturated polytopological spaces\textemdash see especially \cite{Banakh_2018a} and \cite{CCGS_2020a}.  The creative reader may be able to craft interesting problems by combining the machinery of local functions and polytopological spaces.
\end{rem}


\begin{thebibliography}{9}
\bibitem{Banakh_2018a}
T. Banakh, O. Chervak, T. Martynyuk, M. Pylypovych, A. Ravsky, and M. Simkiv, \textit{Kuratowski monoids of $n$-topological spaces,} Topological Algebra and its Applications 6, no. 1 (2018), 1--25.
\bibitem{Bowron_2012a}  M. Bowron, 
\textit{Kuratowski's Closure-Complement Cornucopia} (2012). \url{https://mathtransit.com/cornucopia.php}
\bibitem{CCGS_2020a} S. Canilang, M. P. Cohen, N. Graese, and I. Seong,
{\it The closure-complement-frontier problem in saturated polytopological spaces}, preprint.
\bibitem{Gaida_Eremenko_1974a}  Yu. R. Gaida and A. \'{E}. Eremenko, \textit{On the frontier operator in Boolean algebras with a closure}, Ukr. Math. J. 26.6 (1974), 806--809.
\bibitem{Gardner_Jackson_2008a} B. J. Gardner and M. Jackson, \textit{The Kuratowski closure-complement theorem,}  New
Zealand J. Math. 38 (2008), 9--44.
\bibitem{Kuratowski_1922a} K. Kuratowski, 
{\it Sur l'operation A de l'Analysis Situs}, Fundamenta Mathematicae 3 (1922), 182-–199.
\bibitem{Kuratowski_1966a} K. Kuratowski,
{Topology: Vol. 1}, trans. J. Jaworowski, New York, Academic Press, 1966.
\bibitem{Pettis_1950a} B. J. Pettis,
{\it On continuity and openness of homomorphisms in topological groups}, Ann. of Math. (2) 52, 1950, 293--308.
\bibitem{Rosendal_2009a} C. Rosendal,
{\it Automatic continuity of group homomorphisms}, Bull. Symbolic Logic {\bf 15 (2)} (2009), 184--214.
\bibitem{Shelah_1984a} Shelah, S.,
{\it Can you take Solovay's inaccessible away?}, Israel J. Math. 48(1) (1984), 1--47.
\bibitem{Sherman_2010a} D. Sherman, \textit{Variations on Kuratowski's 14-set theorem}, Amer. Math. Monthly, 117:2 (2010), 113--123.
\bibitem{Solovay_1970a} Solovay, R.M.,
{\it A model of set-theory in which every set of reals is Lebesgue measurable}, Ann. Math. 92 (1970), 1--56.
\end{thebibliography}
\end{document}